\documentclass[12pt,a4paper]{amsart}
\usepackage{amsfonts}
\usepackage[active]{srcltx}
\usepackage{color}

\usepackage{enumerate}
\usepackage[notref,notcite]{showkeys}
\usepackage{amsmath,amssymb,xspace,amsthm}

\newtheorem{theorem}{Theorem}[section]

\newtheorem{remark}[theorem]{Remark}
\newtheorem{lemma}[theorem]{Lemma}
\newtheorem{proposition}[theorem]{Proposition}

\newtheorem{definition}[theorem]{Definition}
\numberwithin{equation}{section} \theoremstyle{definition}

\def\span{\operatorname{span}}

\newcommand{\ZZ}{{\mathbb Z}}

\newcommand{\C}{\ensuremath{\mathbb C}\xspace}

\renewcommand{\a}{\ensuremath{\alpha}}
\renewcommand{\l}{\ensuremath{\lambda}}

\newcommand{\h}{\ensuremath{\mathfrak{H}}}

\newcommand{\Z}{\ensuremath{\mathbb{Z}}\xspace}

\newcommand{\W}{\ensuremath{W}\xspace}

\renewcommand{\phi}{\varphi}

\renewcommand{\leq}{\leqslant}
\renewcommand{\geq}{\geqslant}

\newcommand{\K}{\ensuremath{\mathcal{\K}}\xspace}

\def\mh{\mathfrak{h}}

\def\sl{\mathfrak{sl}}
\def\gl{\mathfrak{gl}}

\def\ot{\otimes}
\def\l{\lambda}
\def\p{\partial}
\def\d{\delta}

\def\S{\mathbb{S}}
\def\K{\mathcal{K}}
\def\F{\mathcal{F}}

\def\L{\mathcal{L}}
\def\tL{\widetilde{\mathcal{L}}}

\def\A{\mathcal{A}}

\def\span{\text{span}}

\def\End{\text{End}}

\def\ker{\text{ker}}

\def\bx{\boxtimes}
\def\wg{\wedge}

\begin{document}
\title[Simple $\S_n$-modules]{Simple modules over the Lie algebras of divergence
zero vector fields on a torus}
\author[B. F. Dubsky, X. Guo, Y. Yao  and K. Zhao]{Brendan Frisk Dubsky, Xianqian Guo, Yufeng Yao  and Kaiming Zhao}
\maketitle

\begin{abstract} \vskip -1cm  Let $n\ge2$ be an integer, $\mathcal{K}_n$ the Weyl algebra over the Laurent polynomial algebra $A_n=\mathbb{C} [x_1^{\pm1}, x_2^{\pm1}, ..., x_n^{\pm1}]$, and $\mathbb{S}_n$  the Lie algebra of  divergence zero vector fields on an $n$-dimensional torus. For any $\mathfrak{sl}_n$-module $V$ and any module $P$ over $\mathcal{K}_n$, we define an $\mathbb{S}_n$-module structure on the tensor product $P\otimes V$. In this paper, necessary and sufficient  conditions for the $\mathbb{S}_n$-modules $P\otimes V$ to be simple are given, and an isomorphism criterion for nonminuscule $\mathbb{S}_n$-modules is provided. More precisely, all nonminuscule $\mathbb{S}_n$-modules are simple, and pairwise nonisomorphic. For minuscule $\mathbb{S}_n$-modules,  minimal and maximal submodules are concretely constructed.

\end{abstract}
\vskip 10pt \noindent {\em Keywords:}  Lie algebra of divergence zero vector fields, Weyl algebra,   irreducible module, (non)minuscule module, Density Theorem.
\vskip 5pt
\noindent
{\em 2010  Math. Subj. Class.:}
17B10, 17B65, 17B66
\vskip 10pt

\section{Introduction}

We denote by $\mathbb{Z}$, $\mathbb{Z}_+$, $\mathbb{N}$ and
$\mathbb{C}$ the sets of  all integers, nonnegative integers,
positive integers and complex numbers, respectively.  All algebras and vector spaces are assumed to be over $\C$.
For any vector space $V$, we denote by $V^*$ the dual space of $V$.
Let $n\ge2$ be a fixed integer throughout this paper.

Representation theory of Lie algebras is a rich  topic attracting extensive attention from many mathematicians. Classification of simple modules is an important step in the study of a module category
over an algebra.
The classification of simple weight modules for several classes of Lie algebras has been achieved over many years of many mathematicians' endeavor. Here we only mention a few such achievements related to the study of the present paper.
Finite-dimensional simple modules for  finite-dimensional semisimple
Lie algebras were classified by Cartan in 1913, see \cite{Ca}. The classification of simple Harish-Chandra modules over the Virasoro algebra was completed by Mathieu in 1992, see [M1].
 The classification of  simple Harish-Chandra
modules over finite dimensional semisimple Lie algebras    was obtained in 2000,
see \cite{M2}.  Billig and Futorny \cite{BF}    classified  simple Harish-Chandra  modules over Witt algebras $W_n$. Besides $\sl_2$ (and some of its deformations), all simple weight modules were classified only for the aging algebra \cite{LMZ},
 the Schr\"odinger algebra \cite{BL2}, and
 the Euclidean algebra \cite{BL1}.

Over the last two decades, the weight representation theory of Witt algebras was  extensively
studied by many mathematicians and physicists; see for example \cite{B, E1, E2, BMZ, GLZ,L3, L4, L5,LiZ,LZ, LLZ,
MZ2,Z}. In 1986,  Shen defined a class of modules $F^\alpha_b(V)$
over the Witt algebra $W_n$ for  $\a\in\C^n$, $b\in\C$,
and a simple  module $ V$ over the special linear Lie algebra
$\sl_n$, see \cite{Sh}, which were also given by Larsson in 1992,
see \cite{L3}. In 1996, Eswara Rao determined the  conditions for   these modules to be irreducible when $V$
is finite dimensional, see \cite{E1, GZ}. Very recently,  Billig and Futorny
\cite{BF}   proved that  simple Harish-Chandra
$W_n$-modules     are modules of the highest weight type or simple quotient
  modules of $F^\alpha_b(V)$. These weight modules over algebraically closed fields of positive characteristic were studied by Nakano \cite{N}.

In the paper \cite{LLZ}, a big class of $W_n$-modules (including weight and nonweight modules)  were constructed and their irreducibility and submodules were determined. More precisely, let $\C^n$ be the natural $n$-dimensional representation of $\sl_n$ and let $V(\delta_k)$ be its
$k$th exterior power, where $k = 0,\cdots,n-1$. It was shown that when $M$ is a weight module over $\sl_n$, $\mathcal{F}(P, M)$ is a simple module over $W_n$ if and only if  $M\not\cong V(\delta_k)$ for any $k\in \{0, 1,\cdots, n\}$ with one more natural condition on $P$ for $k=0$ or $n$. Using these irreducible modules, the paper \cite{GLLZ} completely detemined all
irreducible $W_n$-modules that are finitely generated over the Cartan subalgebra. In the present paper, we will regard the modules in \cite{GLLZ} as  modules over the Lie algebra $\mathbb{S}_n$ of  divergence zero vector fields on an n-dimensional  torus, determine the conditions for them to be irreducible and isomorphic, and also construct minimal and maximal submodules. Besides some irreducible weight $\S_n$-modules which were studied in \cite{T, LGW}, there is a class of nonweight $\S_n$-modules constructed in \cite{Zh}, which are free $U(\mh)$-modules of rank one.

The present paper is organized as follows. In Section 2, we recall related definitions for fundamental representations of $\sl_n$,  the Witt algebra $W_n$, the  algebra $\S_n$, and the Weyl algebra $\K_n$, and furthermore construct our $\S_n$-module $\F( P, V)$ for a $\K_n$-module $P$ and an $\sl_n$-module $V$, also describe some of its properties. In Section 3,     we prove that all nonminuscule $\S_n$-modules are simple  (Theorem \ref{irre}) and pairwise nonisomorphic (Proposition \ref{isomorphism criterion}). Section 4 is devoted to the study of minuscule $\S_n$-modules. Minimal and maximal submodules of minuscule $\S_n$-modules are concretely constructed (Theorem \ref{LLn}).

Our results recover those in \cite{T, LGW} where only  $P=A_n$ was considered. However, our method is very different from the one used in those papers.

\section{$\S_n$-modules from $\K_n$-modules and $\sl_n$-modules}

\subsection{Simple $\sl_n$-modules}

As usual,   $\ZZ^{n}$
denotes the direct sum of $n$ copies of the additive group $\ZZ$, and we consider it as the additive group of all column vectors with integer entries (and similar notation is used for other direct sums). For
any $a=(a_1,\cdots, a_n)^T \in \Z_+^n$ and $m=(m_1,\cdots,m_n)^T
\in\C^n$, we denote $m^{a}=m_1^{a_1}m_2^{a_2}\cdots m_n^{a_n}$,
where $T$ means taking the transpose of the matrix.
Let $\gl_n$ be the Lie algebra of all $n \times n$ complex matrices, and $\sl_n$ be the
subalgebra of $\gl_n$ consisting of all traceless matrices.  For $1
\leq i, j \leq n$, we use $E_{ij}$ to denote the matrix with $1$
at the $(i, j)$ entry and zeros elsewhere. We know that
$$\gl_n=\sum_{1\leq i, j\leq
n}\C E_{i,j}.$$

Let $\h=\span\{h_{i}\,|\,1\le i\le n-1\}$ where
$h_i=E_{ii}-E_{i+1,i+1}$.
Let $\Lambda^+=\{\l\in\h^*\,|\,\l(h_i)\in\Z_+ \text{ for } i=1,2,...,n-1\}$ be the set of dominant weights with respect to $\h$. For any
$\psi\in \Lambda^+$,  let $V(\psi)$ be  the simple $\sl_n$-module with
highest weight $\psi$.

We fix the vector space $\mathbb{C}^n$ of $n\times 1$ matrices.
Denote its standard basis by $\{e_1,e_2,...,e_n\}$. Let
$(\,\cdot\,|\, \cdot\, )$ be the standard symmetric bilinear form  on $\mathbb{C}^n$
such that $(u | v)=u^Tv\in\mathbb{C}$.

Define the fundamental weights $\delta_i\in\h^*$ by
$\delta_i(h_j)=\delta_{i,j}$ for all $i,j=1,2,..., n-1$. For convenience, we also set $\d_0=\d_n=0$.
It is well known that the module $V(\delta_1)$ can be realized as the
natural representation of $\gl_n$ on $\mathbb{C}^n$ (the matrix
product), which we can write as $E_{ji}e_l=\delta_{li}e_j$.  In
particular,
\begin{equation}(ru^T)v=(u|v)r,\,\,\forall\,\, u,v,r\in \mathbb{C}^n.\end{equation}
The exterior
product $\bigwedge^k(\mathbb{C}^n)=\mathbb{C}^n\wedge\cdots\wedge
\mathbb{C}^n\ \ (k\ \mbox{times})$ is a $\mathfrak{gl}_n$-module
with the action given by $$X(v_1\wedge\cdots\wedge
v_k)=\sum\limits_{i=1}^k v_1\wedge\cdots v_{i-1}\wedge
Xv_i\cdots\wedge v_k, \,\,\forall \,\, v_i\in  \mathbb{C}^n,  X\in \gl_n.$$
We set  $\bigwedge^0(\mathbb{C}^n)=\C$ and
$v\wedge a=av$ for any $v\in\C^n, a\in\C$.
The following
$\gl_n$-module isomorphism is well known:
\begin{equation}\label{dk}{\bigwedge}^k(\mathbb{C}^n)\cong V(\delta_k),\,\forall\,\, 0\leq k\leq n.\end{equation}

\subsection{The Lie algebras $W_n$ and $\S_n$}
We denote by  $ W_n$ the derivation Lie algebra of the
Laurent polynomial algebra $A_n=\C[x_1^{\pm1}, \cdots,x_n^{\pm1}]$ which is also called  the Lie algebra of polynomial
vector fields on an n-dimensional torus.
Set $\partial_i=x_i\frac{\partial}{\partial x_i}$ for $i=1,2,\dots,n$ and
$x^r=x_1^{r_1}x_2^{r_2}\cdots x_n^{r_n}$ for $r=(r_1,r_2,\cdots, r_n)^T\in\mathbb{Z}^n$. Then $$W_n={\text{span}}\{x^r\partial_i\mid r\in\mathbb{Z}^n,
1\leq i\leq n\}.$$

For $u=(u_1,\cdots, u_n)^T \in \mathbb {C}^n$ and $r\in \mathbb{Z}^n$, we denote
$$D(u,r)=x^r\sum_{i=1}^nu_i\partial_i\in\W_n.$$ Then we have the Lie bracket
$$[D(u,r),D(v,s)]=D(w,r+s),\ \forall\ u,v\in \mathbb {C}^n, r,s\in \mathbb {Z}^n,$$
where $w=(u | s)v-(v | r)u$. Note that for any $u,v,\xi,\eta\in
\mathbb{C}^n$, both $uv^T$ and $\xi\eta^T$ are $n\times n$ matrices, and
\begin{equation*}(uv^T)(\xi\eta^T)=(v|\xi)u\eta^T.\end{equation*}
We know that $\mh=\span\{\partial_1, \partial_2, ... , \partial_n\}$ is the Cartan
subalgebra of $\W_n$.

Now we recall the Lie subalgebra $\S_n$ of $W_n$:
$$\S_n={\text{span}}\{D(u,r)\mid u\in\C^n, r\in\Z^n {\text{ with }} (u|r)=0\}.$$
This  Lie subalgebra $\S_n$ is called  the Lie algebra of  divergence
zero vector fields on an n-dimensional torus. We see that $\mh$ is its Cartan subalgebra.

\subsection{$\S_n$-modules}
Let $V$ be an $\sl_n$-module and $P$ be a  module over the Weyl algebra $\K_n$.
Let $\F(P,V)=P\otimes V$. We define the  actions of $\S_n$ and $A_n$ on $\mathcal{F}(P, V)$ as follows:
\begin{equation}\label{2.1}
 D(u,r)(z\otimes y)= D(u,r)z\otimes y+ x^rz \otimes(ru^T) y,
\end{equation}
\begin{equation}\label{2.4}
x^r(z\otimes y)=x^rz\otimes y ,
\end{equation}
where $u\in\mathbb{C}^n$, $r\in\mathbb{Z}^n$ with $(r|u)=0$ and $z\in P, y\in V$. Note that $ru^T\in\sl_n$ in \eqref{2.1}. Similarly to \cite{T}, we define the following concept.

\begin{definition}
Let $V$ be a simple $\sl_n$-module and $P$ be a simple $\K_n$-module. The $\S_n$-module  $\F(P,V)=P\otimes V$ is called minuscule if $V\cong V(\delta_k)$ for some $k$ with $0\leq k\leq n-1$. Otherwise, it is called nonminuscule.
\end{definition}

It is easy to see that $\F(P, V)$ is a weight $\S_n$-module if and only if $P$ is a weight $\S_n$-module.
Sometimes we may rewrite (\ref{2.1}) as
\begin{equation}\label{2.5}
(x^r\partial_j)(z\otimes y)= (x^r\partial_j)z\otimes y+\sum_{i=1}^nr_i x^r z \otimes(E_{ij} y),\ {\text if}\ r_j=0.
\end{equation}
Note that the elements $x^r\partial_j, r\in\Z^n, r_j=0$, do not span the whole algebra $\S_n$. These $\S_n$-modules  $\F(P, V)$ become the $W_n$-modules in \cite{GLLZ} if one makes the $\sl_n$-module $V$ into a $\gl_n$-module with any scalar action of the identity matrix.

The following facts should also be noted when we do computations on the module $\F(P,V)$:
we have $x^sD(u,r)p=D(u,r+s)p$ and
\begin{equation}\label{DDp}
D(v,s)D(u,r)p=D(v,r+s)D(u,0)p+(v|r)D(u,r+s)p;
\end{equation}
while we do not have $(ru^T)(sv^T)w=(s|u)(rv^T)w$,
for $u,v\in\C^n, r,s\in\Z^n$ and $p\in P, w\in V$.

\section{Nonminuscule $\S_n$-modules $\F(P,V)$}

In this section we prove that all nonminuscule  $\S_n$-modules $\F(P,V)$ are simple. First we recall a result from \cite{LZ}.

\begin{lemma}\label{sln module}
Let $L$ be an irreducible $\sl_n$-module (not necessarily a weight module).
\begin{itemize}
\item [(a)] For any $i, j$ with $1\leq i\neq j\leq n$, $E_{ij}$ acts injectively or
   locally nilpotently on $L$.
\item [(b)] The module $L$ is finite-dimensional if and only if $E_{ij}$ acts locally
   nilpotently on $L$ for any $i, j$ with $1\leq i\neq j\leq n$.
\end{itemize}
\end{lemma}
 The following result will be widely used.
\begin{theorem}\label{density}
Let $M$ be an irreducible module over a (unital) associative algebra $\A$ which has countable basis.
Then for any linearly independent elements $v_1,\cdots,v_m\in M$ and any elements $u_1,\cdots,u_m\in M$,
there exists some $a\in \A$ such that $av_i=u_i$ for $i=1,\cdots,m$.
\end{theorem}

\begin{proof} First by a result similar to Schur's Lemma (see [Di] Page 87, Proposition 2.6.5), we have
$\text{End}_{\A}(M)=\C$.

Then we recall the Density Theorem in ring theory (see [J], Page197):
Let $S$ be a completely reducible module for a ring $R$ and $R'=\text{End}_R(S)$,
$R^{''}=\text{End}_{R'}(S)$. Let $\{w_i|i\in I\}$ be a finite subset of $S$ and $a^{''}\in R^{''}$.
Then there exists an $a\in R$ such that $aw_i=a^{''}w_i$, for all $i\in I$.

In our case, we let $\A'=\C$ and $\A''=\End_{\C}(M)$. There exists $a''\in\A''$ such that
$a''(v_i)=u_i$ for $i=1,\cdots,m$. Consequently, the statement follows from the Density Theorem.
\end{proof}

For any $r=(r_1,\cdots,r_n)\in\Z^n$ and $i=1,\cdots,n-1$, we denote $D_{i,r}=D(r_{i+1}e_i-r_ie_{i+1},r)\in\S_n$.
Then we have
$$D_{i,r}(p\ot w)=D_{i,r}p\ot w+x^rp\ot \sum_{l=1}^nE_{l,i,r}w,$$
for any $p\in P$ and $w\in V$, where $E_{l,i,r}=r_lr_{i+1}E_{l,i}-r_lr_{i}E_{l,i+1}$.

We are now in the position to present the following main result in this section, which asserts that nonminuscule $\S_n$-modules are irreducible.
\begin{theorem}\label{irre}
The $\S_n$-module $\F(P,V)$ is irreducible if $P$ is an irreducible $\K_n$-module and
$V$ is an irreducible $\sl_n$-module not isomorphic to any $V(\delta_k)$ for $k=0, 1,\cdots,n-1$.
\end{theorem}

\begin{proof}
For any  $r,s\in\Z^n$, $p\ot w\in\F(P,V)$, $1\le i, j\le n-1$, we have
\begin{equation}\label{DDpw}
\begin{aligned}
&D_{j,s-r}D_{i,r}(p\ot w)\\
=& D_{j,s-r}\Big(D_{i,r}p\ot w+x^rp\ot \sum_{l=1}^nE_{l,i,r}w\Big)\\
=&  D_{j,s-r}D_{i,r}p\ot w+x^{s-r}D_{i,r}p\ot \sum_{l'=1}^nE_{l',j,s-r}w\\
& +D_{j,s-r} x^rp\ot \sum_{l=1}^nE_{l,i,r}w+x^sp\ot \sum_{l'=1}^nE_{l',j,s-r}\sum_{l=1}^nE_{l,i,r}w\\
=& D_{j,s-r}D_{i,r}p\ot w+D(r_{i+1}e_i-r_ie_{i+1},s)p\ot \sum_{l'=1}^nE_{l',j,s-r}w\\
& +D((s_{j+1}-r_{j+1})e_j-(s_j-r_j)e_{j+1},s) p\ot \sum_{l=1}^nE_{l,i,r}w\\
& +(r_js_{j+1}-r_{j+1}s_j)x^{s} p\ot \sum_{l=1}^nE_{l,i,r}w\\
& +x^sp\ot \sum_{l'=1}^nE_{l',j,s-r}\sum_{l=1}^nE_{l,i,r}w.
\end{aligned}
\end{equation}
Regarding the above expression as a polynomial in $r\in\Z^n$, with coefficients in $P\ot V$,
we can see that only the last summand in the RHS of (\ref{DDpw}) contains degree $4$ terms.
More precisely, we have
$$\aligned
& x^sp\ot \sum_{l'=1}^nE_{l',j,s-r}\sum_{l=1}^nE_{l,i,r}w\\
=& x^sp\ot\bigg(\sum_{l=1}^n\big((s_{l'}-r_{l'})(s_{j+1}-r_{j+1})E_{l',j}-(s_{l'}-r_{l'})(s_j-r_j)E_{l',j+1}\big)\\
&\cdot\sum\limits_{l=1}^n\big(r_lr_{i+1}E_{l,i}-r_lr_{i}E_{l,i+1}\big)w\bigg)\\
= & x^sp\ot \sum_{l,l'=1}^nE_{l,j,r}E_{l',i,r}w+ {\text { (terms\ of\ degree\ less\ than}} \ 4 {\text {  in }} r).
\endaligned$$

Suppose that $N$ is a nonzero $\S_n$-submodule of $\F(P,V)$ and take any nonzero element
$y=\sum_{k\in I}p_k\ot w_k\in N$, where $I$ is a finite index set and $p_k, k\in I$ are linearly independent.

\smallskip
\noindent{\bf Claim 1.} $\sum_{k\in I}x^sp_k\ot E_{t',t}E_{t',t}w_k\in N$, for all $s\in\Z^n, 1\leq t'\neq t\leq n$.
\smallskip

By applying $D_{j,s-r}D_{i,r}$ on $y$, we have
\begin{equation}\label{DDy}
D_{j,s-r}D_{i,r}(y)=\sum_{k\in I}x^sp_k\ot \sum_{l,l'=1}^nE_{l,j,r}E_{l',i,r}w_k+y'\in N,
\end{equation}
where $y'$ is an element in $\F(P,V)$ with degree less than $4$ when regarded as a polynomial in $r\in\Z^n$.
It follows from the property of the Vandermonde determinant that
\begin{equation*}
\sum_{k\in I}x^sp_k\ot \sum_{l,l'=1}^nE_{l,j,r}E_{l',i,r}w_k\in N,
\end{equation*}
or equivalently,
\begin{equation*}
\sum_{k\in I}x^sp_k\ot \sum_{l,l'=1}^n(r_{l'}r_{j+1}E_{l',j}-r_{l'}r_{j}E_{l',j+1})(r_lr_{i+1}E_{l,i}-r_lr_{i}E_{l,i+1})
w_k\in N.
\end{equation*}
Consider the above expression as a polynomial in $r_1,\cdots,r_n$ with coefficients in $P\ot V$.
If $t'>t$, then by taking $i=j=t$ and considering the coefficients of $r^2_{t'}r^2_{t+1}$ there, we can deduce the claim.
If $t'<t$, taking $i=j=t-1$ and considering the coefficients of $r^2_{t'}r^2_{t-1}$, again we deduce the claim.

\smallskip
\noindent{\bf Claim 2.}  $E_{t',t}E_{t',t}w_k\neq0$ for $k\in I$, and some $1\leq t'\neq t\leq n$.
\smallskip

If $V$ is infinite dimensional, by Lemma \ref{sln module} there exist  some $1\leq t'\neq t\leq n$ such that $E_{t',t}$ acts injectively on $V$. The claim is obvious.

Now we suppose that $V$ is finite-dimensional, and hence a highest
weight module $V(\l)$ with $\l$ being the dominant highest weight.
Let $v$ be a nonzero weight component which has maximal weight among all homogeneous components of all $w_k, k\in I$. If the weight of $v$ is not $\l$, then there exists $1\leq j\leq n-1$ such that $E_{j,j+1}v$ is nonzero, and has higher weight.
Now we consider that
\begin{equation}\label{Djk}
\begin{aligned}
&D_{j,r}\sum_{k\in I}p_k\ot w_k\\
=&\sum_{k\in I}\big(D_{j,r}p_k\ot w_k+x^{r}p_k\ot\sum_{l=1}^n(r_lr_{j+1}E_{l,j}-r_lr_{j}E_{l,j+1})w_k\big)\\
\end{aligned}
\end{equation}
Taking suitable $r\in\Z^n$ such that $r_j\gg r_i$ for $i\neq j$, we see that $-r_j^2E_{j,j+1}v$ is a nonzero weight
component of some $\sum_{l=1}^nE_{l,j,r}w_k, k\in I$ with weight higher than that of $v$.
Replacing $\sum_{k\in I}p_k\ot w_k$ with $D_{j,r}\sum_{k\in I}p_k\ot w_k$, and repeating the above process several times, we may assume that the weight of $v$ is $\l$.

Now suppose $\l=\sum_{d=1}^{n-1}a_d\delta_d$ for some $a_d\in\Z_+$.
Since $\l\neq \delta_d$ for any $d=1,\cdots,n-1$, there exist some $1\leq
d_1\leq d_2\leq n-1$ such that $a_{d_1}+a_{d_2}\geq2$. Let
$\mathfrak{s}$ be the $3$-dimensional simple Lie algebra spanned by
$E_{d_1,d_2+1}, E_{d_2+1,d_1}$ and $E_{d_1,d_1}-E_{d_2+1,d_2+1}$. Note that
$\l(E_{d_1,d_1}-E_{d_2+1,d_2+1})=a_{d_1}+a_{d_1+1}+\cdots+a_{d_2}$, hence the
$\mathfrak{s}$-module generated by $v$ has highest weight
$a_{d_1}+a_{d_1+1}+\cdots+a_{d_2}\geq 2$. In particular, $E_{d_2+1,d_1}^2v\neq 0$ and
$E_{d_2+1,d_1}^2w_k\neq 0$ for some $k\in I$.  The claim follows.
\smallskip

Using Claim 1 and Claim 2, and replacing $w_k$ with $E_{t',t}E_{t',t}w_k$ for suitable $t', t$,
we may assume that $$\sum_{k\in I}x^sp_k\ot w_k\in N,\ \forall\ s\in\Z^n.$$
Moreover, applying $\partial_i$ on $\sum_{k\in I}p_k\ot w_k$, we also have
$$\sum_{k\in I}\partial_i p_k\ot w_k\in N,\ \forall\ i=1,\cdots,n.$$
As a result, we have
$$\sum_{k\in I}y p_k\ot w_k\in N,\ \forall\ y\in\K_n.$$

By Theorem \ref{density}, we can deduce that $p_1\ot w_1\in N\setminus\{0\}$.
By the previous arguments, we have $\K_n p_1\ot w_1\subseteq N$.
Since $P$ is a simple $\K_n$-module, we get $P\ot w_1\subseteq N$.
Denote $$V'=\{w\in V\ | \ P\ot w\subseteq N\}.$$
Then $V'\neq0$. Moreover, $V'$ is stable under the action of $\sl_n$. Consequently, $V'=V$ and $P\ot V\subseteq N$, as desired.
\end{proof}

The following result gives an isomorphism criterion for nonminuscule $\S_n$-modules $\F(P,V)$.
\begin{proposition}\label{isomorphism criterion}
Let $P, P^{\prime}$ be irreducible $\K_n$-modules, and $V,V^{\prime}$ be irreducible ${\mathfrak{sl}}_n$-modules. Suppose that $V\not\cong V(\delta_k)$ for $k=0,1,\cdots, n-1$. Then the two $\S_n$-modules $\F(P,V)$ and $\F(P^{\prime}, V^{\prime})$ are isomorphic if and only if $P\cong P^{\prime}$ and $V\cong V^{\prime}$.
\end{proposition}

\begin{proof}
The sufficiency is obvious. Next we prove the necessity. For this sake, suppose that
$$\varphi:\F(P, V)\longrightarrow \F(P^{\prime}, V^{\prime})$$
is an isomorphism of $\S_n$-modules. Let $p\ot v$ be a nonzero element in $\F(P, V)$, and assume that
$$\varphi(p\ot v)=\sum\limits_{i=1}^m p_i\ot v_i,$$
where $p_1,\cdots, p_t$ are linearly independent and $v_i\ne0$. By rechoosing $p\ot v$,
since $V\not\cong V(\delta_k)$ for $k=0,1,\cdots, n-1$, without loss of generality, we can assume that $E_{t^{\prime},t} E_{t^{\prime},t}v_1\neq 0$ for some distinct $t^{\prime}, t\in\{1,2,\cdots n\}$.
As in Claim 1 in the proof of Theorem \ref{irre}, we have
\begin{equation}\label{iso expression}
\varphi(yp\ot E_{t^{\prime},t} E_{t^{\prime},t}v)=\sum\limits_{i=1}^m yp_i\ot  E_{t^{\prime},t} E_{t^{\prime},t}v_i,\forall y\in\K_n, 1\leq t\neq t^{\prime}\leq n.
\end{equation}
 Moreover, we note that since $p_1,\cdots, p_t$ are linearly independent, it follows from Theorem \ref{density} that there exists $u\in\K_n$ such that $up_i=\delta_{i1}p_i$ for $i=1,2,\cdots, t$. Let $y=u$ in (\ref{iso expression}), we have
\begin{equation*}
\varphi(up\ot E_{t^{\prime},t} E_{t^{\prime},t}v)=p_1\ot  E_{t^{\prime},t} E_{t^{\prime},t}v_1\neq 0,
\end{equation*}
which implies that $up\ne0$ and $E_{t^{\prime},t} E_{t^{\prime},t}v\neq 0$. Now in (\ref{iso expression}), replace $y$ with $yu$, then replace $up$ with $p$, $E_{t^{\prime},t} E_{t^{\prime},t}v$ with $v$, $E_{t^{\prime},t} E_{t^{\prime},t}v_1$ with $v^{\prime}$,
we then get
\begin{equation}\label{key expression1}
\varphi(yp\ot v)=yp_1\ot v^{\prime}, \forall y\in\K_n.
\end{equation}
Since $\varphi$ is an isomorphism, (\ref{key expression1}) implies that Ann$_{\K_n}(p)=$Ann$_{\K_n}(p_1)$. It follows that $P\cong\K_n/ \text{Ann}_{\K_n}(p)=\K_n/ \text{Ann}_{\K_n}(p_1)\cong P^{\prime}$. Moreover, the map $\varphi_1: P\longrightarrow P^{\prime}$ with
$\varphi_1(yp)=yp_1$ gives the isomorphism, where $y\in\K_n, p\in P$. Hence
\begin{equation}\label{key expression2}
\varphi(p\otimes v)=\varphi_1(p)\otimes v^{\prime}.
\end{equation}

For any distinct $i, j\in\{1,2,\cdots n\}$, $p\in P$, we have
\begin{equation}\label{S_n morphism}
\varphi\big((x_i\p_j)(p\ot v)\big)=(x_i\p_j)\varphi(p\ot v).
\end{equation}
It follows from (\ref{S_n morphism}) and (\ref{key expression2}) that
\begin{equation*}
\varphi(p\ot E_{ij}v)=\varphi_1(p)\ot E_{ij}v^{\prime}.
\end{equation*}
Consequently,
\begin{equation*}
\varphi(p\ot uv)=\varphi_1(p)\ot uv^{\prime},\,\forall\,p\in P, u\in U({\mathfrak{sl}}_n).
\end{equation*}
This implies that  Ann$_{U({\mathfrak{sl}}_n)}(v)=$Ann$_{U({\mathfrak{sl}}_n)}(v^{\prime})$. Since $V$ and $V^{\prime}$ are simple $U({\mathfrak{sl}}_n)$-modules, it follows that $V\cong U({\mathfrak{sl}}_n)/ \text{Ann}_{U({\mathfrak{sl}}_n)}(v)=U({\mathfrak{sl}}_n)/ \text{Ann}_{U({\mathfrak{sl}}_n)}(v^{\prime})\cong V^{\prime}$, as desired.
\end{proof}

\section{Minuscule $\S_n$-modules $\F(P,V(\delta_k))$}
In this section, we study the structure of the minuscule $\S_n$-modules
$\F(P,V(\delta_k))$ for $k=0,1,\cdots,n-1$. First we recall some known results from \cite{LLZ} and \cite{GLLZ}.

We make $V(\delta_k)$ into a $\gl_n$-module with the identity matrix acting as a scalar multiplication by $k$ and hence $\F(P,V(\delta_k))$ becomes a $W_n$-module. For each $0\leq k\leq n-1$, we can define the natural linear map:
\begin{equation*}\begin{array}{crllc}
d_k: & \F(P, V(\delta_k)) &  \longrightarrow & \F(P, V(\delta_{k+1}))&\\
& p \otimes w& \mapsto& \sum\limits_{i=1}^n(\partial_ip)\otimes (e_i\wedge w),& \forall\ p\in P, w\in V({\delta_k}),
\end{array}\end{equation*}
where $V(\delta_n)$ is the irreducible $\gl_n$ module with highest weight $\delta_n$ and the action of the identity matrix
is the scalar multiplication by $n$.
It is straightforward to check that all $d_k$ are $W_n$-module homomorphisms and
$d_{k+1}d_{k}=0$. Thus we obtain the following generalized de Rham complex of $W_n$-modules:
$$0\hskip -3pt \to\hskip -3pt \F(P, V(\delta_0))\hskip -3pt \to\hskip -3pt  \F(P, V(\delta_1))\hskip -3pt \to\hskip -3pt \cdots \F(P, V(\delta_{n-1}))\hskip -3pt \to \hskip -3pt \F(P, V(\delta_n))\hskip -3pt \to\hskip -3pt  0.$$

In \cite{LLZ}, the authors studied another $W_n$-module structure on $P\otimes V$ defined as follows:
\begin{equation}\label{Action2}
 (x^{r-e_j}\partial_{j})\circ  (p\otimes w)
=((x^{r-e_j}\partial_{j})p)\otimes w+ \sum_{i=1}^nr_i(x^{r-e_i}p)\otimes E_{ij}(w),
\end{equation}
for all $r=(r_1,\cdots,r_n)^T\in \Z^n, p\in P$ and $w\in V.$
Denote this module by $F(P, V)$. For such modules, one can also define the natural homomorphisms:
\begin{equation*}\begin{array}{crllc}
\pi_k: & F(P, V(\delta_k)) &  \longrightarrow & F(P, V(\delta_{k+1}))\\
& p \otimes w& \mapsto& \sum\limits_{i=1}^n(x^{-e_i}\partial_ip)\otimes (e_i\wedge w), \forall\ p\in P, w\in V({\delta_k}).
\end{array}\end{equation*}
and the generalized de Rham complex:
$$0\hskip -3pt \to \hskip -3pt F(P, V(\delta_0))\hskip -3pt \to \hskip -3pt  F(P, V(\delta_1))\hskip -3pt \to\hskip -3pt \cdots F(P, V(\delta_{n-1}))\hskip -3pt \to \hskip -3pt F(P, V(\delta_n))\hskip -3pt \to\hskip -3pt  0.$$
The two classes of $W_n$-homomorphisms and generalized de Rham complexes defined above are essentially the same. Let us explain this.

For any $\lambda=(\l_1,\cdots,\l_n)\in \C^n$, we have the automorphism $\tilde{\lambda}$ of $\K_n$ defined by
$$\tilde{\l}(x^{\a})=x^{\a}, \quad \tilde{\l}(\partial_j)=\partial_j- {\lambda}_j,\ \forall\ j=1,2,\ldots,n.$$
Then for any module $P$ over the associative algebra $\K_n$, we have the new module $P^{\tilde{\lambda}}=P$ with the action
$$y \circ_{\lambda} p={\tilde{\lambda}}(y) p,\ \forall\ y\in \K_n,\ p\in P^{\tilde{\lambda}}.$$
If  $V$ is a simple weight $\gl_n$-module, there is a $\lambda\in\C^n$ such that
\begin{equation}\label{weightset}
V=\bigoplus_{\mu\in \Z^n}V_{\lambda+\mu},
\end{equation}
where $V_{\lambda+\mu}=\{  v\in V\mid E_{ii}v=(\lambda_i+\mu_i) v, \ \forall\ i=1, 2, \cdots, n\}$.

\begin{lemma}([GLLZ, Theorem 2.3])\label{GLLZ iso} If $V$ is a simple weight module over $\gl_n$ with decomposition
(\ref{weightset}) for some $\l\in\C^n$ and $P$ is a module over the associative algebra $\K_n$.
Then the linear  map
$$\aligned
\varphi: \F(P, V) & \longrightarrow& F(P^{\tilde{\lambda}}, V)\hskip5pt&\\
p \otimes v_\mu& \mapsto& x^{-\mu}p \otimes v_\mu,&\quad\forall\ v_\mu\in V_{\lambda+\mu},\ p\in P
\endaligned$$
is a $W_n$-module isomorphism.
\end{lemma}

\begin{remark}
For the simple weight  $\gl_n$-module $V=V(\delta_k)$ with $k=0,\dots, n$, we can take $\lambda=(0,\cdots,0)^T$. In this case, $\tilde{\lambda}$ is the identity automorphism, and the $\K_n$-module structure of $P^{\tilde{\lambda}}$ coincides with that of $P$.
\end{remark}

The following result, implicitly used in \cite{GLLZ}, can be verified directly.

\begin{lemma}\label{natural equiv} The de Rham complexes defined by $\{d_k, k=0,\cdots,n\}$ and $\{\pi_k, k=0,\cdots,n\}$
are naturally equivalent. More precisely, we have the following commutative diagram of $W_n$-modules:
\begin{equation*}\begin{array}{clc}
\F(P, V(\delta_k)) & \xrightarrow{d_k} & \F(P, V(\delta_{k+1}))\\
\downarrow{\phi_k} &  & \downarrow{\phi_{k+1}}\\
F(P, V(\delta_k)) &  \xrightarrow{\pi_k} & F(P, V(\delta_{k+1})),
\end{array}\end{equation*}
where $\phi_k$ is the $W_n$-module isomorphism defined in Lemma \ref{GLLZ iso} with $V=V(\d_k)$, $\l=(0,\cdots,0)^T$,
and similarly for $\phi_{k+1}$.
\end{lemma}

Certainly,  the above results still hold if we regard the $W_n$-modules as modules over $\S_n$.
For $1\leq k\leq n-1$, denote by $\L_n(P,k)$ the image of $d_{k-1}$, Moreover, we set $\L_n(P,0)=0$.
It follows from a straightforward calculation that
$$\L_n(P,k)=\span\Big\{\sum_{l=1}^n(\p_l p)\otimes E_{lj}w, p\in P, w\in V(\d_k), 1\leq i\leq n\Big\}.$$
For $0\leq k\leq n-1$, we let
$$\tL_n(P,k)=\{y\in \F(P, V(\delta_k))  \ | \ W_n y\subseteq \L_n(P,k)\}.$$
Both $\L_n(P,k)$ and $\tL_n(P,k)$ are $W_n$-submodules of $\F(P,V(\d_k))$,
and hence also $\S_n$-submodules of $\F(P,V(\d_k))$.
Furthermore we can replace $W_n$ with $\S_n$ in the definition of $\tL_n(P,k)$. The following result asserts this.

\begin{lemma}\label{equiv def} For any $k=0,1,\cdots,n-1$, we have
\begin{itemize}\item[(a)] $\tL_n(P,k)=\{y\in \F(P, V(\delta_k))  \ | \ \S_n y\subseteq \L_n(P,k)\}$;
\item[(b)] $\tL_n(P,k)=\{v\in \F(P, V(\delta_k))  \ | \ \S_n v\subseteq \tL_n(P,k)\}$.\end{itemize}
\end{lemma}

\begin{proof} (a) Suppose the result  is not true.
We can choose some $y\in \F(P, V(\delta_k))\setminus\tL_n(P,k)$ such that $\S_ny\subseteq \L_n(P,k)$.
Then we have $d_k(y)\neq0$ and $\S_nd_k(y)=0$ in $\L_n(P,k+1)$.
In particular, $\p_ld_k(y)=0$ for all $l=1,\cdots,n$ in $\F(P,V(\d_{k+1}))$.
Suppose that $d_k(y)=\sum_{i\in I}p_i\ot w_i$, where $I$ is a finite index set, $p_i\in P, i\in I$, are all nonzero and $w_i\in V(\d_{k+1}), i\in I$, are linearly independent. Then we have
$$\p_ld_k(y)=\sum_{i\in I}\p_lp_i\ot w_i=0,\ \forall\ l=1,\cdots,n.$$
This implies $\p_lp_i=0$ for all $l=1,\cdots,n$ and $i\in I$ and hence
$P$, as a simple $\K_n$ module, is a weight module with all weights in $\Z^n$.
Moreover, we have $P\cong A_n$, the natural module. In this case we know that  $ \L_n(A_n, k+1)$ is a simple $W_n$-module which does not have zero weight vectors (see Page 2369 in \cite{GZ}), contradicting the facts
$d_k(y)\neq0$ and $\S_nd_k(y)=0$. Part (a) follows.

(b) Suppose the result  is not true.
We can choose  $v\in \F(P, V(\delta_k))\setminus\tL_n(P,k)$ such that $\S_nv\subseteq \tL_n(P,k)\setminus\L_n(P,k)$. Then $ \S_n\S_n v\subseteq \L_n(P,k)$ and we have $d_k(\S_nv)\neq0$ and $\mh d_k(\S_nv)=0$ in $\L_n(P,k+1)$.
 As in the arguments in the previous paragraph, we deduce that $P\cong A_n$, the natural module. In this case we know that  $\L_n(A_n, k+1)$ is a simple $W_n$-module which does not have zero weight vectors,
contradicting the facts $d_k(\S_nv)\neq0$ and $\mh d_k(\S_nv)=0$. Part (b) follows.
\end{proof}

\begin{remark}
In the proof of Lemma \ref{equiv def}, we actually proved that
$$\tL_n(P,k)=\{y\in \F(P, V(\delta_k))  \ | \ \p_l y\subseteq \L_n(P,k),\ \forall\ l=1,\cdots,n\}$$
for $k=0,1,\cdots,n-1$.
\end{remark}

The following result for $\F(P, V(\delta_k))$  directly follows from the corresponding one for $F(P,V(\d_k))$ in \cite[Lemma 3.4, Theorem 3.5]{LLZ} and the relationship between  $\F(P,V(\d_k))$ and $F(P,V(\d_k))$ given in Lemma \ref{natural equiv}.

\begin{lemma}\label{known LLZ}
Keeping the notation as before, we have
\begin{itemize}
\item[(a)] $\tL_n(P,k)={\rm{ker}}(d_{k})$ for $0\leq k\leq n-1$;
\item[(b)] $\L_n(P,k)$ is a proper $\S_n$-submodule of $\F(P,V(\d_k))$ for $1\leq k\leq n-1$;
\item[(c)] The $\S_n$-module $\F(P, V(\delta_k))$ is not irreducible for $1\leq k\leq n-1$.
\end{itemize}
\end{lemma}

Next we will first study the $\S_n$-modules $\F(P,V(\d_0))$, $\F(P,V(\d_1))$, and then study other $\F(P,V(\d_k))$ for $n>2$.

\begin{lemma}\label{F'(P,0)}
$\mh\F(P,V(\d_0))$ is an $\S_n$-submodule of $\F(P,V(\d_0))$ and the quotient $\F(P,V(\d_0))/\mh\F(P,V(\d_0))$ is trivial.
\end{lemma}

\begin{proof} We only need to show that $\S_n\F(P,V(\d_0))\subseteq \mh\F(P,V(\d_0))$.
For  any $p\in P$, $u\in\C^n, r\in\Z^n$ with $(u|r)=0$, we have
$$\aligned
  D(u,r)(p\ot 1)
&=D(u,r)p\ot 1=x^rD(u,0)p\ot 1=D(u,0)x^rp\ot 1,
\endaligned$$
which lies in $\mh\F(P,V(\d_0))$, as desired.
\end{proof}

Note that for $P=A_n$, the natural module, we have $A_n=A'_n\oplus\C$ as $\S_n$-modules,
where $A'_n=\sum_{r\in\Z^n\setminus\{0\}}\C x^r$ is a simple $\S_n$-submodule of $A_n$
and $\C=\C x^0$ is the trivial $\S_n$-submodule of $A_n$.
Hence we have $\F(P,V(\d_0))=\mh\F(P,V(\d_0))\oplus (\C\ot V(\d_0))$ as $\S_n$-modules,
where $\mh\F(P,V(\d_0))=A'_n\ot V(\d_0)\cong A'_n$
and $\C\ot V(\d_0)\cong \C$.

\begin{proposition}\label{k=0}
The following statements hold.
\begin{itemize}
\item[(a)]
If $P\not\cong A_n$, then $\F(P,V(\d_0))$ has a unique simple $\S_n$-submodule
$\mh\F(P,V(\d_0))$ and any nonzero submodule is of the form $P'\ot V(\d_0)$, where $P'$ is a subspace with
$\mh P\subseteq P'\subseteq P$;
\item[(b)]If $P\cong A_n$, then $\F(P,V(\d_0))\cong A'_{n}\oplus\C$.\end{itemize}
In particular, $\F(P,V(\d_0))$ is a simple $\S_n$-module  if and only if $\mh P=P$.
\end{proposition}

\begin{proof} As before we make the identification $V(\d_0)=\C$.
Take a nonzero $\S_n$-submodule $N\subseteq \F(P,V(\d_0))$ and a nonzero element $p\ot 1\in N$.

For any $i=1,2,\cdots,n-1$, considering the coefficients of $-r_i^2$ and $-r_{i+1}^2$ in \eqref{DDpw} with $w=1$, we get:
$$x^s\p_{i+1}\p_{i+1}p\ot 1+s_{i+1}x^s\p_{i+1}p\ot 1=\p_{i+1}x^s\p_{i+1}p\ot 1\in N, \ \forall\ s\in\Z^n,$$
and
$$x^s\p_i\p_ip\ot 1+s_ix^s\p_ip\ot 1=\p_ix^s\p_ip\ot 1\in N,\ \forall\ s\in\Z^n.$$
We divide the following discussion into three cases.

{\it Case (i): $\p_ip\neq0$ for any $i=1,\cdots,n$.}

In this case, $\p_i\K_n\p_ip\ot 1\subseteq N$, so that $\p_iP\ot 1\subseteq N$ for any $i=1,\cdots,n$.

{\it  Case (ii): $\p_ip=0$ and $\p_jp\neq0$ for some $i\neq j$.}

In this case, $\p_ix_i\p_jp=x_i\p_jp\neq0$. Then applying the previous argument to $x_i\p_jp\ot 1\in N$, we have
$\p_iP\ot 1\subseteq N$. Combining with Case 1 we   obtain that $\p_kP\ot 1\subseteq N$ for all $k=1,\cdots,n$.

{\it  Case (iii): $\p_ip=0$ for any $i=1,\cdots,n$.}

In this case, $P$ is a simple weight $\K_n$-module, and hence isomorphic to the natural module $A_n$. We may assume that $p=1$. The result for this case is quite clear.

Obviously, (a) follows from Case (i) and (ii); (b) follows from Case (iii).
\end{proof}

\begin{remark}\label{L_0} The result in Proposition \ref{k=0} for $\S_n$ is quite different from that for $W_n$.  Actually as a $W_n$-module, $\F(P,V(\d_0))$ is irreducible if and only if  $P\not\cong A_n$ (see Theorem 3.5 (4) in \cite{LLZ}).
\end{remark}

For any $k=1,\cdots,n-1$, denote
$$p\bx w=d_{k-1}(p\ot w)=\sum_{l=1}^n\p_l p\ot e_l\wg w,$$
for any $p\in P$ and $w\in V(\d_{k-1})=\bigwedge^{k-1}\C^n$.
Since $d_{k-1}$ is a $W_n$-module homomorphism, we have
$$D(u,r)(p\bx w)=D(u,r)p\bx w+x^rp\bx (ru^T)w.$$
It is clear that $$\aligned &\L_n(P,k) =\span\{p\bx w\ |\ p\in P, w\in V(\d_{k-1})\},
\\ &\mh\L_n(P,k) =\mh P\bx V(\d_{k-1}).\endaligned$$

Similarly to Lemma \ref{F'(P,0)}, we have
\begin{lemma}\label{L_2'}
$\mh \L_n(P,1)$ is an $\S_n$-submodule of $\L_n(P,1)$ and the quotient module $\L_n(P,1)/\mh \L_n(P,1)$ is trivial.
\end{lemma}

\begin{proof} We only need to show that $\S_n\L_n(P,1)\subseteq\mh \L_n(P,1)$.
For any $p\in P$, $u\in\C^n, r\in\Z^n$ with $(u|r)=0$, we have
$$\aligned
  D(u,r)(p\bx 1)
&=D(u,r)p\bx 1=x^rD(u,0)p\bx 1=D(u,0)x^rp\bx 1,
\endaligned$$
which lies in $\mh \L_n(P,1)$, as desired.
\end{proof}

\begin{remark}\label{L_n' and L_n}
Note that $\mh \L_n(P,1)=\L_n(P,1)$ for all weight $\K_n$-modules $P$.
Indeed, for any weight modules $P$, we can define the weight spaces
$$P_{\l}=\{p\in P\ |\ D(u,0)p=(\l|u)p,\ \forall\ u\in\C^n\},\ \forall\ \l\in\C^n.$$
Note that for $p\in P_\l$, we have $p\bx 1=(\l|u)^{-1}D(u,0)p\bx 1$ for suitable $u\in\C^n$ if $\l\neq0$
and $p\bx 1=\sum_{l=1}^n\p_l p\ot e_l=0$ if $\l=0$.
Hence both $\L_n(P,1)$ and $\mh \L_n(P,1)$ are equal to the space
$$\span\{D(u,0)p\bx w \ |\ u\in\C^n, w\in V(\d_{k-1}), p\in P_\l, \l\in\C^n\setminus\{0\}\}.$$
There are both examples of irreducible $\K_n$-modules $P$ such that  $\mh P=P$ and examples where this is not the case. See Examples in \cite{TZ2}.
\end{remark}

As a consequence of Proposition \ref{k=0}, we have
\begin{proposition} \label{k=1}The following statements hold.
\begin{itemize}\item[(a)]
If $P\not\cong A_n$,   $\L_n(P,1)$ has a unique simple $\S_n$-submodule $\mh \L_n(P,1)$.
\item[(b)]If $P\cong A_n$, then $\L_n(P,1)=\mh \L_n(P,1)\cong A'_n$ as $\S_n$-modules.\end{itemize}
In particular, $\L_n(P,1)$ is simple if and only if $\mh P=P$ or $P=A_n$.
\end{proposition}

\begin{proof}
We note that $\ker d_0=\{p\ |\ \p_ip=0\ \forall\ i=1,\cdots,n\}$, which is nonzero if and only if
$P$ is a simple weight $\K_n$-module with all weights in $\Z^n$, that is, if and only if $P\cong A_n$.
If $P\not\cong A_n$, then $d_0$ is an isomorphism, and the statement (a) follows immediately from Proposition \ref{k=0} (a).
If $P\cong A_n$, then $\ker d_0=\C$. Then it follows from Proposition \ref{k=0} (b) and Remark \ref{L_n' and L_n} that $\mh \L_n(P,1)=\L_n(P,1)\cong A'_n$ as $\S_n$-modules.
\end{proof}

\begin{remark}\label{L_1} The result in Proposition  \ref{k=1} for $\S_n$ is quite different from that for $W_n$.  Actually, as a $W_n$-module, $\L_n(P,1)$ is always irreducible  (see Theorem 3.5 (3) in \cite{LLZ}).
\end{remark}

For $s\in\Z^n, 1\le i\le n-1, 1\le k\le n$,
define the linear map   $g_{i,s}$ on $\F (P, V(\delta_{k}))$ as follows
$$\aligned g_{i,s}(p\ot w)=&x^s\p_{i+1}p\ot E_{i,i+2}w-x^s\p_{i+2}p\ot E_{i,i+1}w\\
&+\sum_{l=1}^nx^s\p_lp\ot E_{l,i+2}E_{i,i+1}w,\ \forall\ p\in P, w\in V(\delta_{k}) .\endaligned$$
We have
\begin{lemma} \label{linear maps f, g} Let $1\le i\le n-1, s\in\Z^n$. Then the restriction of the linear map $ g_{i,s}$ on $\L_n(P,k)$ is identically zero for $1\le k\le n$.
\end{lemma}

\begin{proof} Take any $\sum_{l=1}^n\p_lp\ot E_{lj}w\in \L_n(P,k)$ with $w=e_{j_1}\wedge\cdots\wedge e_{j_k}$ for some distinct $1\leq j_1=j,j_2,\cdots,j_k\leq n$.
We have the following  calculations:\begin{equation}\label{equation for g}
\begin{aligned}
&g_{i,s}(\sum_{l=1}^n\p_lp\ot E_{lj}w)\\
=&\sum_{l=1}^n\left(x^s\p_{i+1}\p_lp\ot E_{i,i+2}E_{lj}w-x^s\p_{i+2}\p_lp\ot E_{i,i+1}E_{lj}w\right)\\
&+\sum_{l=1}^n\sum_{l'=1}^nx^s\p_{l'}\p_lp\ot E_{l',i+2}E_{i,i+1}E_{l,j}w.
\end{aligned}
\end{equation}
The term involving $x^s\p_{i+1}^2p$ in (\ref{equation for g}) is
$$x^s\p_{i+1}^2p\ot (E_{i,i+2}E_{i+1,j}+E_{i+1,i+2}E_{i,i+1}E_{i+1,j})w=0.$$
The term involving $x^s\p_{i+2}^2p$ in (\ref{equation for g}) is
$$x^s\p_{i+2}^2p\ot (E_{i+2,i+2}E_{i,i+1}E_{i+2,j}-E_{i,i+1}E_{i+2,j})w=0.$$
The term involving $x^s\p_{i+1}\p_{i+2}p$ in (\ref{equation for g}) is
\begin{eqnarray*}
x^s\p_{i+1}\p_{i+2}p\ot (E_{i,i+2}E_{i+2,j}-E_{i,i+1}E_{i+1,j}+E_{i+1,i+2}E_{i,i+1}E_{i+2,j}\\
+E_{i+2,i+2}E_{i,i+1}E_{i+1,j})w=0.\hspace{6cm}
\end{eqnarray*}
The term involving $x^s\p_l\p_{i+1} p$ in (\ref{equation for g}) for $ l\neq i+1, i+2$ in  is
$$x^s\p_l\p_{i+1} p\ot (E_{i,i+2}E_{lj}+E_{i+1,i+2}E_{i,i+1}E_{l,j}+E_{l,i+2}E_{i,i+1}E_{i+1,j})w=0.$$
The term involving$x^s\p_l\p_{i+2} p$ in (\ref{equation for g}) for $ l\neq i+1, i+2$  is
$$x^s\p_l\p_{i+2} p\ot (-E_{i,i+1}E_{lj}+E_{i+2,i+2}E_{i,i+1}E_{l,j}+E_{l,i+2}E_{i,i+1}E_{i+2,j})w=0.$$
The term involving $x^s\p^2_l p$ in (\ref{equation for g}) for $ l\neq i+1, i+2$   is
$$x^s\p^2_l p\ot E_{l, i+2}E_{i, i+1}E_{l,j}w=0.$$
The term involving$x^s\p_l\p_{l'} p$ in (\ref{equation for g}) for $ l\neq i+1, i+2, l^{\prime}\neq i+1, i+2$  is
$$x^s\p_l\p_{l'} p\ot (E_{l^{\prime},i+2}E_{i,i+1}E_{lj}+E_{l,i+2}E_{i,i+1}E_{l^{\prime},j})w=0.$$
Hence the RHS of (\ref{equation for g}) is zero, i.e., the restriction of
the linear map $g_{i,s}$ on $\L_n(P,k)$ is 0.
\end{proof}

The following result asserts the simplicity of $\L_n(P,k)$ as an $\S_n$-module for any $n\geq 3$ and $2\leq k\leq n-1$.
\begin{proposition}\label{Ln}
Let $n\geq 3$ and $2\leq k\leq n-1$. Then $\L_n(P,k)$ is a simple $\S_n$-submodule of $\F(P, V(\delta_k))$ for any simple $\K_n$-module $P$. Consequently,
$\tL_n(P,k)$ is a maximal $\S_n$-submodule of $\F(P,V(\d_k))$ for  any simple $\K_n$-module $P$.
\end{proposition}

\begin{proof} Let   $r,s\in\Z^n$, $p\ot w\in\F(P,V(\d_k))$, $1\le i, j\le n-1$.
Taking $j=i+1$ in \eqref{DDpw}, the coefficient of $-r_i^2$ is
\begin{equation}\label{r_i^2}\aligned
& s_{i+1}\Big(x^s\p_{i+1}p\ot E_{i,i+2}w-x^s\p_{i+2}p\ot E_{i,i+1}w\\
&+\sum_{l=1}^nx^s\p_lp\ot E_{l,i+2}E_{i,i+1}w-\sum_{l=1}^n\p_l x^sp\ot E_{l,i+2}E_{i,i+1}w\Big)\\
=&s_{i+1}g_{i,s}(p\ot w)-s_{i+1}\sum_{l=1}^n\p_l x^sp\ot E_{l,i+2}E_{i,i+1}w.
\endaligned\end{equation}

Suppose that $N$ is a nonzero $\S_n$-submodule of $\L_n(P,k)$.
Take any nonzero $y=\sum_{j\in J}p_j\ot w_j\in N$, where $J$ is a finite index set, all $w_j\in V(\d_k), j\in J$, are nonzero and all $p_j\in P, j\in J$, are linearly independent.
Let $v$ be a nonzero weight component which has minimal weight among all homogeneous components of all $w_j, j\in J$.

\smallskip
\noindent{\bf Claim 1:} We can choose $y$ so that $v\in\C e_{n-k+1}\wg\cdots\wg e_n$.
\smallskip

If $v\notin\C e_{n-k+1}\wg\cdots\wg e_n$, i.e., the weight of $v$ is not $\l_k=\d_n-\d_{n-k}$, the lowest weight of $V(\d_k)$, then there exists $1\leq d\leq n-1$ such that $E_{d+1,d}v$ is nonzero, and has lower weight. Now we consider the equation \eqref{Djk} with $j, k$ therein replaced by $d, j$ respectively:
\begin{equation}\label{Ddj}
\begin{aligned}
&D_{d,r}\sum_{j\in J}p_j\ot w_j\\
=&\sum_{j\in J}\big(D_{d,r}p_j\ot w_j+x^{r}p_j\ot\sum_{l=1}^n(r_lr_{d+1}E_{l,d}-r_lr_{d}E_{l,d+1})w_j\big).\\
\end{aligned}
\end{equation}
Taking suitable $r\in\Z^n$ such that $r_{d+1}\gg r_{d'}$ for any $d'\neq d+1$, we see that there exists some $j\in J$ such that $r_{d+1}^2E_{d+1,d}v$ is a nonzero weight component of $\sum_{l=1}^nE_{l,d,r}w_j$ with weight lower than that of $v$, and $x^rp_j\otimes r_{d+1}^2E_{d+1,d}v$ can not be canceled by other summands.
Replacing $\sum_{k\in I}p_k\ot w_k$ with $D_{j,r}\sum_{k\in I}p_k\ot w_k\neq0$, and repeating this process several times, we may assume that the weight of $v$ is $\l_k$. That is, $v\in\C e_{n-k+1}\wg\cdots\wg e_n$,  Claim 1 follows.

\smallskip

Assume that $v$ is a nonzero weight component of some $w_{j_0}, j_0\in J$.
Since $n\geq 3$ and $2\leq k\leq n-1$, we have $E_{l,n-k+2}E_{n-k,n-k+1}v\neq0$ and $E_{l,n-k+2}E_{n-k,n-k+1}w_{j_0}\neq0$
for some $l$ with $1\leq l\leq n$.

Now by applying the action of $D_{i+1,s-r}D_{i,r}$ on $y=\sum_{j\in J}p_j\ot w_j$ and using \eqref{r_i^2}, we   obtain that
$$s_{i+1}g_{i,s}(y)-s_{i+1}\sum_{j\in J}\sum_{l=1}^n\p_l x^sp_j\ot E_{l,i+2}E_{i,i+1}w_j\in N.$$
Since $y\in N\subseteq \L_n(P,k)$, it follows from Lemma \ref{linear maps f, g} that
$$s_{i+1}\sum_{j\in J}\sum_{l=1}^n\p_l x^sp_j\ot E_{l,i+2}E_{i,i+1}w_j\in N,$$
i.e.,
\begin{equation}\label{elements in N}
\sum_{j\in J}\sum_{l=1}^n\p_l x^sp_j\ot E_{l,i+2}E_{i,i+1}w_j\in N,\forall s\in\Z^n {\text{ with}}  \ s_{i+1}\neq0.
\end{equation}

\smallskip
\noindent{\bf Claim 2:} For any nonzero $p\in P$, there is $s'\in\Z^n$  (depending on $p$) and $w\in V(\delta_{k-1})$ (independent of $p$) such that
$$  (x^{s+s'}p)\boxtimes w \in N\setminus\{0\},\ \forall\ s\in\Z^n {\text{ with }}  s_{n-k+1}\geq 0.$$
\smallskip

Note that all $p_j, j\in I$, are linearly independent.
By Theorem \ref{density},  we can find some elements $z\in\K_n$ such that
$zp_{j_0}=p$ and $zp_{j}=0$ for all $j\neq j_0$. Now take $s'\in\Z^n$ such that
$z'=x^{s'}z\in\sum_{r'_{i+1}>0}\C[\p_1,\cdots,\p_n]x^{r'}$.  Then we have
$z'p_{j_0}=x^{s'}p\ne0$ and $z'p_{j}=0$ for all $j\neq j_0$.
It follows from (\ref{elements in N}) that
$$\sum_{l=1}^n\p_l x^sz'p_{j_0}\ot E_{l,i+2}E_{i,i+1}w_{j_0}\in N,\ \forall\ s\in\Z^n {\text{ with }}   s_{i+1}\geq 0.$$
Taking $i=n-k$, we get
$$
\sum_{l=1}^n\p_l x^{s+s'}p\ot E_{l,n-k+2}E_{n-k,n-k+1}w_{j_0}\in N,\ \forall\ s\in\Z^n {\text{ with }}   s_{n-k+1}\geq 0.
$$
Claim 2 follows.

\smallskip

Let $V'$ be the subset of $V(\d_{k-1})$ consisting of all elements $w$ that satisfies the following conditions:
for any $p\in P$, there exists $s'\in\Z^n$ such that $(x^{s+s'}p)\bx w\in N$ for all $s\in\Z^n$ with $s_{n-k+1}\geq0$.
From Claim 2 we see that $V'\neq0$ is a subspace of $V(\d_{k-1})$.

Take any $w\in V', p\in P$ and $s'\in\Z^n$ such that $x^{s+s'}p\bx w\in N$ for all $s\in\Z^n$ with $s_{n-k+1}\geq0$.
Then for any $u\in\C^n, r\in\Z^n$ with $(u|r)=0$, for any  $s\in\Z^n$ (depending on $p$, $w$ and $r$) satisfying $s_{n-k+1}\geq0$ and $s_{n-k+1}+r_{n-k+1}\geq0$, we have
\begin{equation}\label{V'}\aligned
&D(u,r)(x^{s+s'}p\bx w)\\
=&(x^rD(u,0)x^{s+s'}p)\bx w+(x^{r+s+s'}p)\bx (ru^T)w\\
=&D(u,0)(x^{s+s'+r}p\bx w)+(x^{s+s'+r}p)\bx (ru^T)w\in N,
\endaligned\end{equation}
which implies $(x^{s+s'+r}p)\bx (ru^T)w\in N$, that is, $(ru^T)w\in V'$.
Hence $V'$ is a nonzero $\sl_n$-submodule of $V(\d_{k-1})$, forcing $V'=V(\d_{k-1})$.

Take $s\in\Z^n$ such that $s_{n-k+1}\geq0$ and $s_{n-k+1}+s'_{n-k+1}\geq0$.
Take $r=-(s+s')$ in \eqref{V'} and $u\in\C^n$ with $(u|r)=0$, we get
$$D(u,0)(p\bx w)+p\bx (ru^T)w\in N.$$
Replacing $r$ with $2r$ and $u$ with $2u$ in above formula (all the requirements are still satisfied), we get
$$2D(u,0)(p\bx w)+4p\bx (ru^T)w\in N.$$
Combining these two formulas, we obtain that $p\bx (ru^T)w\in N$.
Note that $s'$ is fixed and $r=-(s+s')$ only need to satisfy the conditions
$s_{n-k+1}\geq0$ and $s_{n-k+1}+s'_{n-k+1}\geq0$. The elements of the form $ru^T$ with $(u|r)=0$ can span the algebra $\sl_n$.
We obtain that $P\bx V(\d_{k-1})\subseteq N$. Hence, $\L_n(P,k)$ is a simple $\S_n$-module, as desired.

The second statement follows directly from the first one.
\end{proof}

Now we can summarize the results in this section into the following theorem.

\begin{theorem}\label{LLn} Let $P$ be a simple $\K_n$-module. Then the following statements hold.
\begin{itemize}
\item[(a)] The $\S_n$-module $\F(P,V(\d_0))$ is irreducible if and only if $\mh P=P$.
\item[(b)] The $\S_n$-module $\F(P, V(\delta_k))$ is not irreducible for $1\leq k\leq n-1$.
\item[(c)] If $P\not\cong A_n$, then $\F(P,V(\d_0))$ has a unique simple $\S_n$-submodule
$\mh\F(P,V(\d_0))$, and any nonzero submodule is of the form $P'\ot V(\d_0)$, where $P'$ is a subspace with
$\mh P\subseteq P'\subseteq P$. If $P\cong A_n$, then $\F(P,V(\d_0))\cong A'_{n}\oplus\C$.
\item[(d)] If $P\not\cong A_n$, $\L_n(P,1)$ has a unique simple $\S_n$-submodule $\mh \L_n(P,1)$. If $P\cong A_n$, then $\L_n(P,1)=\mh \L_n(P,1)\cong A'_n$
as $\S_n$-modules. In particular, $\L_n(P,1)$ is simple if and only if $\mh P=P$ or $P=A_n$.
\item[(e)] For  $n\geq 3$ and $2\leq k\leq n-1$,  the $\S_n$-module $\L_n(P,k)$ is simple for any simple $\K_n$-module $P$. Consequently,
$\tL_n(P,k)$ is a maximal $\S_n$-submodule of $\F(P,V(\d_k))$.
\end{itemize}
\end{theorem}

\begin{center}
\bf Acknowledgments
\end{center}

The research was carried out during the visit of the first two authors to Wilfrid
Laurier University in the summer of 2017. The hospitality of Wilfrid Laurier University is gratefully acknowledged.

Brendan Frisk Dubsky is partially supported by the Lundstr\"om-\AA{}man foundation. 

Xianqian Guo is partially supported by NSF of China (Grant No. 11471294) and the Outstanding Young Talent Research Fund.

Yufeng Yao is partially supported by NSF of China (Grant Nos. 11771279, 11571008 and 11671138) and NSF of Shanghai (Grant No. 16ZR1415000).

Kaiming Zhao is partially supported by NSF of China (Grant No. 11271109) and NSERC.

\
\

\noindent Brendan Frisk Dubsky: Department of Mathematics, Uppsala University, Box 480, SE-75106, Uppsala,
Sweden. E-mail: brendan.frisk.dubsky@math.uu.se \vspace{1mm}

\noindent Xianqian Guo: School of Mathematics and Statistics, Zhengzhou University,
Zhengzhou, 730000 P. R. China. Email: guoxq@zzu.edu.cn \vspace{1mm}

\noindent Yufeng Yao: Department of Mathematics, Shanghai Maritime University,
Shanghai, 201306 P. R. China. Email: yfyao@shmtu.edu.cn \vspace{1mm}

\noindent Kaiming Zhao: College of Mathematics and Information Science, Hebei Normal (Teachers)
University, Shijiazhuang, Hebei, 050016 P. R. China, and Department of Mathematics, Wilfrid
Laurier University, Waterloo, ON, Canada N2L 3C5. Email: kzhao@wlu.ca

\end{document}